\documentclass{amsart}
\usepackage[T2A]{fontenc}
\usepackage[cp1251]{inputenc}
\usepackage[english]{babel}
\usepackage{amsfonts,amssymb,mathrsfs,amscd}
\usepackage{enumitem}
\usepackage{comment}
\usepackage{relsize}
\usepackage{graphicx}
\usepackage{comment}
\usepackage{atbegshi}
\theoremstyle{plain}
\newtheorem{theorem}{Theorem}
\newtheorem*{hypothesis}{Conjecture}
\newtheorem{lemma}{Lemma}
\newtheorem{propos}{Proposition}
\theoremstyle{definition}
\newtheorem{definition}{Definition}
\newtheorem{example}{Example}

\newcommand\blfootnote[1]{%
  \begingroup
  \renewcommand\thefootnote{}\footnote{#1}%
  \addtocounter{footnote}{-1}%
  \endgroup
}
\begin{document}
\newpage
\setcounter{page}{1}
	\title{On Christophersen's problem}
\author{Roman  Stasenko}
\address{HSE University, Faculty of Computer Science, Pokrovsky Boulevard 11, Moscow,
109028 Russia}
\email{theromestasenko@yandex.ru}
\maketitle

\begin{abstract}
Let  $A$ be a finite-dimensional (Artinian) Gorenstein algebra, and let $\operatorname{Aut}(A)^{\circ}$ denote the connected component of the identity in the automorphism group of  $A$.  We introduce a new subclass of Gorenstein algebras and prove that for any algebra  $A$ in this subclass, the group $\operatorname{Aut}(A)^{\circ}$ is solvable. This result is closely related to the Christophersen problem in the theory of local algebras.
\end{abstract}

\blfootnote{The article is prepared within the framework of the project ``International Academic Cooperation'' HSE University”

2020 Mathematics Subject Classification. Primary 13H10, 14M05; Secondary 22E45, 20G05.

{\it Key words and phrases.} {Local algebra, derivation, Lie algebra, Gorenstein algebra, socle.}}
\section{Introduction}\label{intro}
Let $\mathbb{C}$ be our base field. A central object in the study of finite-dimensional commutative algebras is a local algebra. By definition, a unital, commutative, associative $\mathbb{C}$-algebra  $A$ is {\it local} if it contains precisely one maximal ideal.

Local algebras are some kind of "briks" for every finite dimensional commutative associative algebra. Namely, every finite dimensional commutative associative algebra $A$ is a direct sum of a finite number of local ideals $A_i\subseteq A$, i.e. the decomposition
	$$A = A_1\oplus A_2\oplus ...\oplus A_l$$
	comes true for some $l\in \mathbb{Z}_{>0}$. The proof of this fact is presented in \cite[Section~1,~Lemma 1.3]{iarz}.

One can also easily prove that $A$ is local if and only if it has the form $A = \mathbb{C}\oplus\mathfrak{m}$ \cite[Section 1,~lemma 1.2]{iarz}, where the first summand is one-dimensional vector subspace of $A$ spanned by the unit and the second summand is the nilpotent ideal of $A$, which is  the maximal ideal of $A$ . 

It is well known, that every finite dimensional local algebra $A$ over $\mathbb{C}$ is isomorphic to the factor algebra  $\mathbb{C}[x_1,...,x_m]/I$, where $I$ is ideal in algebra of polynomials over $\mathbb{C}$, which is generated by  polynomials $f_1,...,f_k$, such that $f_i(0,...,0) = \frac{\partial f_i}{\partial x_j}(0,...,0) = 0,\quad\forall i=\overline{1, k}, j=\overline{1,m}$ . In addition we may suppose that among the generators $f_1,...,f_k$ there are monomials $x_i^{k_i}$ for every $i=\overline{1,m}$ and some $k_i\in\mathbb{Z}_{> 0}$.

 	The classification of all local algebras of fixed dimension is rather difficult problem. It  is still not completed. There are only finite number of local algebras of dimension less than 7. Starting from the dimension 7 there are infinitely many local algebras. Many results, which are connected with such classification, were obtained by D. Suprunenko and R.Tyshkevich in \cite{suptish}. 
 
 	Consider an arbitrary local algebra  $A = \mathbb{C}\oplus \mathfrak{m}$ of dimension $n$, where $\mathfrak{m}$ is the maximal ideal of $A$. The group $\operatorname{Aut}(A)$ of automorphisms of the algebra $A$ is an algebraic subgroup in the group $GL(A)$ of non-degenerate linear transformations on $A$.  
 	
 	 	 Let us give two canonical examples.
 	 
 	 \begin{example}\label{ex1}
 	 	Let $A = \mathbb{C}[x_1,...,x_{n-1}]/I,\quad  I = (x_i^2, x_ix_j: i,j=\overline{1, n-1}, i\neq j)$. It is easy to see that 
 	 	$$A = \mathbb{C}\oplus \langle \overline{x}_1,..., \overline{x}_{n-1}\rangle,$$
 	 	where $\overline{x}_i = x_i+I$.
 	 	One can compute that in this case $\operatorname{Aut}(A) = GL_{n-1}(\mathbb{C})$. 
 	 	
 	 \end{example}
 \begin{example}\label{ex2}
	Consider the local algebra $A = \mathbb{C}[t]/(t^n)$ for some $n\in\mathbb{Z}_{> 0}.$ It is clear that $$A = \mathbb{C}\oplus\langle \overline{t}, \overline{t}^2,..., \overline{t}^{n-1}\rangle,\quad\overline{t}^i = t^i +(t^n).$$
	Of course, every automorphism $\phi\in\operatorname{Aut}(A)$ is defined uniquely  by the element $\phi(\overline{t})$. So, if we fix the image  $\phi (\overline{t})$ as $\phi(\overline{t}) = a_1\overline{t} + a_2\overline{t}^2 + \dots +a_{n-1}\overline{t}^{n-1}, a_1\neq 0$, then all other values $\phi(a)$, where $a\in A$ can be computed immediately. Consequently, here we have $\dim\operatorname{Aut}(A) = n-1.$
\end{example} 
 	Christophersen's problem can be formulated as the following:
 	\begin{hypothesis}
 			 For every local algebra $A$ it is true that $\dim{\operatorname{Aut}(A)}^{\circ}\geqslant n-1$.  Futhermore, a local algebra $A$ with the property $\dim{\operatorname{Aut}(A)}^{\circ} = n-1$ is isomorphic to $ \mathbb{C}[t]/(t^n)$.
 	\end{hypothesis}

 	Of course, this conjecture can be reformulated in terms of the Lie algebra $\operatorname{Lie}(\operatorname{Aut}(A)) = \operatorname{Der}(A)$. Here we have
 	$$\operatorname{Der}(A) = \left\{\xi\in \mathfrak{gl}(A), \xi(ab) = \xi(a)b+a\xi(b),\quad\forall a,b\in A\right\}.$$

Consider an arbitrary $\xi\in\operatorname{Der}(A)$ and $0\neq a\in\mathfrak{m}$. Assume that $\xi(a) = s\cdot 1 + m$, where $s\in\mathbb{C}^* = \mathbb{C}\backslash\{0\}$, $1$ is the unit of  $A$ and $m\in \mathfrak{m}$ (so  $\xi(a)\notin \mathfrak{m}$).  Then there exists such $b\in A$, that $b\xi(a) = 1$. The ideal $\mathfrak{m}$ is nilpotent, so there is such $k\in\mathbb{Z}_{>0}$ that $a^k\neq 0$ and $a^{k+1} = 0$. Thus we have
$$(k+1)a^k = (k+1)a^kb\xi(a) = b\xi(a^{k+1}) = 0,$$
so we can derive that $a^k = 0$ and obtain a contradiction. Consequently, the maximal ideal $\mathfrak{m}$ is invariant under the action of every derivation $\xi\in\operatorname{Der}(A)$. We can canonically identify every element of $\operatorname{Der}(A)$ with its restriction to the maximal ideal $\mathfrak{m}$ and think that $\operatorname{Der}(A)\subset\mathfrak{gl}(\mathfrak{m})$.

 Consequently, we  have that
$$\dim\operatorname{Der}(A)\leqslant (n-1)^2$$
for an arbitrary local algebra $A$. From Example \ref{ex1} we can obtain that this bound is reached.

 	 	\begin{definition} (define grading)
 		A commutative algebra $A$ is called {\it non-negatively graded},  if there is a decomposiiton
 		$$A = A_0\oplus \overset{k}{\underset{i=1}{\bigoplus}}A_i,$$
 		where $A_iA_j\subset A_{i+j}$ for arbitrary $i,j=\overline{1,k}$ and $A_0 = \mathbb{C}\cdot 1$.
 	 	\end{definition}
 	
 	  Consider the set $$\operatorname{Soc}(A)= \{a\in\mathfrak{m}\colon am=0,\quad \forall m\in\mathfrak{m} \} .$$ Clearly, the set  $\operatorname{Soc}(A)$ is a nilpotent ideal in $A$, which is called  {\it the socle of $A$}. 
 	 	
 	In 1996 S.Yau in \cite{Yau} gave a lower bound for dimension of  $\operatorname{Der}(A)$, where $A$ is non-negatively graded local algebra.

 	\begin{propos}[Yau] 
 		For non-negatively graded local algebra $A$ it follows that $\dim\operatorname{Der}(A)\geqslant n-\dim \operatorname{Soc}(A). $
 	\end{propos}

 	Further in 2014  in \cite{Perep} A. Perepechko gave another bound, which holds for every local algebra, not only for non-negatively graded.
 	
 	\begin{propos}[Perepechko]
 		For every local algebra $A$ the following inequality holds
 		$$\dim\operatorname{Der}(A)\geqslant \dim (\mathfrak{m}/\mathfrak{m}^2)\cdot \dim\operatorname{Soc}(A).$$
 	\end{propos}
 	In \cite{Perep} the author used the result of M. Schulze (see \cite{Shulze}) about local algebras, which was proved earlier in 2009. This result can be formulated as the following criterion of solvability of  $\operatorname{Der}(A)$.
 	\begin{theorem}[Schulze]
 		Let $A = \mathbb{C}[x_1,\dots,x_k]/I$ be finite-dimensional local algebra, where $I\subset \mathfrak{p}^l, \mathfrak{p}=(x_1,...,x_k)$. If the inequality
 		$$\dim (I/\mathfrak{p}I)<k+l-1$$
 		holds, then $\operatorname{Der}(A)$ is solvable.
 	\end{theorem}
 	
 	The question of solvalbility of $\operatorname{Der}(A)$ for some local algebra $A$ is the main subject of the work.	During our discusion we give some examples of  local algebras of special kind and consider the connection between properties of  a local algebra $A$ and properties of the Lie algebra $\operatorname{Der}(A)$.

 	Let us present the structure of this paper.
 	
 	Section \ref{prem} is devoted to basic definitions and statements in the theory of local algebras, that we need to prove the main result. The aim of this section is to introduce the new subclass of the class of Gorenstein algebras, which will be considered further.
 	
 	In Section \ref{sec2} the main result of the paper is proved in Theorem \ref{main}. This theorem claims that for every algebra $A$ from introduced in Section \ref{prem} subclass of Gorenstein algebras of full null-index the Lie algebra  $\operatorname{Der}(A)$  is solvable.

 	 	\section{Preliminaries}\label{prem}
 	 	
 	 	Let $A=\mathbb{C}\oplus\mathfrak{m}$ be an arbitrary finite-dimensional local algebra with maximal ideal $\mathfrak{m}$. Let $r$ be the index of nilpotency of $\mathfrak{m}$. Consider the following filtration:
 	 	$$A=\mathfrak{m}^0\supset \mathfrak{m}\supset\mathfrak{m}^2\supset\mathfrak{m}^3\supset...\supset\mathfrak{m}^r\neq 0, \quad \mathfrak{m}^{r+1}=0.$$
 	 	Put $k_i = \dim(\mathfrak{m}^{i+1}/\mathfrak{m}^i),\quad i\in\{0,1....,,r-1\}$. Remark that $\mathfrak{m}^{r}\subset\operatorname{Soc}(A)$.

 	 	\begin{definition}
 	 		The tuple $(k_0, k_1, k_2,...., k_{r-1}) $ is called the {\it Hilbert-Samuel sequence}. 
 	 	\end{definition} 	 	
 	 	A natural question directly connected with Hilbert-Samuel sequences, is to detect such tuples  $(k_0,...,k_{r-1})$ that there are only finitely many  non-isomorphic local algebras with the Hilbert-Samuel sequence $(k_0,k_1,...,k_{r-1})$. Nowadays, only several results are known for some special cases of Hilbert-Samuel sequences. For example, in the paper \cite{Loginov} the author considered the case $k_1=2$.

 	 Let $\mathfrak{g}=\operatorname{Der}(A)$ be the Lie algebra of derivations of  $A$. Then $\operatorname{Soc}(A)$ is a $\mathfrak{g}$-module. Indeed, for  arbitrary $\xi\in\mathfrak{g}, s\in \operatorname{Soc}(A), m\in\mathfrak{m}$ we have
 	 \begin{equation*}
 	 		 	\xi(s)m = \xi(sm) - s\xi(m) = 0,\quad \xi(m)\in\mathfrak{m},
 	 \end{equation*}
 so $\xi(s)\in\operatorname{Soc}(A)$  for all $s\in\operatorname{Soc}(A)$. Using Levi's theorem we have the decomposion
 \begin{equation}
 	\mathfrak{g} = \mathfrak{g}_L\overset{\rightarrow}{\oplus}\operatorname{rad}(\mathfrak{g}),
 \end{equation}
 	 where $\mathfrak{g}_L\subset\mathfrak{g}$ is a semisimple Levi subalgebra, $\operatorname{rad}(\mathfrak{g})$ is the solvable radical of $\mathfrak{g}$ and the symbol $\overset{\rightarrow}{\oplus}$ denotes the semidirect sum with the ideal summand on the right-hand side. Clearly, if  $\mathfrak{g}$ is not solvable, then $\mathfrak{g}_L\neq 0$, so there are three linearly independent elements  $e, f, h \in \mathfrak{g}_L$ such that 
 	 \begin{equation}
 	 	[e, f] =h,\quad [h, f] = -2f,\quad[h, e] = 2e.
 	 \end{equation}
 The triple $T=(e, f, h)$ is called an {\it $\mathfrak{sl}_2$-triple} and its span $\mathfrak{a}=\langle e, f, h \rangle \subset \mathfrak{g}$ is isomorphic to the Lie algebra $\mathfrak{sl}_2$. Then according to the representation theory of $\mathfrak{sl}_2$, we have
 \begin{equation}\label{grad}
 	\mathfrak{g} = \mathfrak{g}^{-k}\oplus\dots\oplus\mathfrak{g}^{-1}\oplus\mathfrak{g}^0\oplus\mathfrak{g}^{1}\oplus\dots\oplus\mathfrak{g}^{k}
 \end{equation}
 for some $k\in\mathbb{Z}_{\geqslant 0}$, where $$\mathfrak{g}^l = \{\xi\in\mathfrak{g}\colon [h, \xi] = l\xi\}, $$ for each $l\in \mathbb{Z}, -k\leqslant l\leqslant k$ is the corresponding  eigenspace of  $h$ in $\mathfrak{g}$. The subspace $$\mathfrak{g}_{*} = \underset{-k\leqslant l \leqslant k, l\neq 0}{\bigoplus}\mathfrak{g}^l$$ is nonzero because $e\in\mathfrak{g}^2\subset \mathfrak{g}_{*}$. Consequently, $k\neq 0$.  We will call the subspace $\mathfrak{g}_{*}$ the {\it irrelevant subspace} of  $\mathfrak{g}$.
 	  	 	\begin{definition}
 	 		The algebra $A$ is called {\it Gorenstein} if $\dim(\operatorname{Soc}(A)) = 1$.
 	 		 	 	\end{definition}
 	 
Now assume that $A$ is Gorenstein.  We are going to introduce a subclass of Gorenstein algebras, which will be the main object of our paper.

 Put $V=\mathfrak{m}/\mathfrak{m}^2$.  Then $\mathfrak{g}$ acts on $V$ by the linear transformations with respect to the linear representation
 $$\rho\colon\mathfrak{g}\rightarrow \mathfrak{gl}(V), \quad \rho(\xi)(m+\mathfrak{m}^2) = \xi(m)+\mathfrak{m}^2,\qquad\forall\xi\in\mathfrak{g}, m\in\mathfrak{m}.$$
Fix an arbitrary non-zero element $s\in\operatorname{Soc}(A)$, i.e. $\operatorname{Soc}(A) = \langle s\rangle$.  Then there exists such a homogeneous polynomial function $P\in\mathbb{C}[V]$ of degree $r$ that $$m^r= P(v)s,\quad \pi(m) = v, \quad\forall m\in\mathfrak{m},$$
where $\pi: \mathfrak{m}\rightarrow V$ is the canonical projection. It is easy to see that $P$ is well-defined. For an arbitrary class $v=m+\mathfrak{m}^2 $ in $V$ and its two representative elements $m$ and $m'$ one can compute that $m^r = m'^r$. 

The polynomial $P$ is homogeneous  because for an arbitrary $\lambda\in\mathbb{C}$ and $v\in V$, such that  $v = \pi(m)$, where $m\in \mathfrak{m}$, we have
$$P(\lambda v)\cdot s =  (\lambda m)^r =\lambda^r\cdot m^r = \lambda^r P(v)\cdot s. $$ 

We  call the polynomial $P$ {\it the index polynomial} of $A$ and denote it by $P_A$.
\begin{definition}
	Let $A$ be a Gorenstein algebra. If for every nonzero $\xi\in\operatorname{Der}(A)$ and for an arbitrary eigenvalue $\lambda\in\operatorname{Spec}(\xi)$ there is such an eigenvector  $v\in V_{\lambda}$ that $P_A(v)\neq 0$, we will call $A$ {\it  the algebra of full null-index}.
\end{definition}
Consider again Example \ref{ex2}. In this case we have
$$\mathfrak{m} = \langle \overline{t}, \overline{t}^2,..., \overline{t}^{n-1}\rangle,\quad\operatorname{Soc}(A) = \langle\overline{t}^{n-1}\rangle,\quad V\backsimeq\langle\overline{t}\rangle,\quad r=n-1,\quad\overline{t}^i = t^i +(t^n).$$
 One can compute that here $P_A(x) = x^{n-1}.$
	As we can see, $P_A(v)\neq 0$ for all $v\in V$, so this algebra is the algebra of full null-index.
	
We give another two examples.
\begin{example}
  	Put $A=\mathbb{C}[t, s]/ I,$ where $ I  = (ts, t^3 - s^3).$  We have
	$$\mathfrak{m} = \langle \overline{t},\overline{s},\overline {t}^2, \overline{s}^2, \overline{t}^3\rangle,\quad\operatorname{Soc}(A) = \langle \overline{t}^3\rangle,\quad V\backsimeq \langle \overline{t},\overline{t}\rangle,\quad r=3,\quad \overline{t} = t+ I, \overline{s} = s + I.$$
	The index polynomial in this case has the view $$P_A(x_1, x_2) = x_1^3 + x_2^3.$$
	If we consider an arbitrary $\xi\in\operatorname{Der}(A)$ and compute the matrix $X$ of linear transformation $\rho(\xi)$ in the basis $\{\overline{t}, \overline{s}\}$, then we will get the equality
	$$X = \begin{pmatrix}
	\lambda&0\\
	0&\lambda
	\end{pmatrix}$$
	for some $\lambda \in\mathbb{C}$. So, every nonzero vector $v\in V$ is an eigenvector of  $\rho(t)$ or $\rho(s)$ with the eigenvalue $\lambda$. If we take $v = \overline{t}$, we have $P_A(v)\neq 0$, so here $A$ is of full null-index.
\end{example}
\begin{example}
	Consider the 5-dimensional Gorenstein algebra $$A = \mathbb{C}[t, s]/(ts, t^3-s^2).$$ One checks that $$\mathfrak{m} = \langle \overline{t}, \overline{s}, \overline{t}^2, \overline{s}^2\rangle,\quad\operatorname{Soc}(A) = \langle\overline{s}^{2}\rangle,\quad V\backsimeq\langle\overline{t},\overline{s}\rangle,\quad r=3.$$The index polynomial here has the formula  $$P_A(x_1, x_2) = x_1^{3}.$$ 
	The matrix $X$ of the linear transformation $\rho(\xi)\in\mathfrak{gl}(V),$ where $\xi\in\operatorname{Der}(A)$, in the basis $\{\overline{t} + \mathfrak{m}^2, \overline{s}+\mathfrak{m}^2\}$ of  $V$, has the following view
	$$X = \begin{pmatrix}
		\alpha&0\\
		\beta&\frac{3}{2}\alpha
	\end{pmatrix}$$
	for some $\alpha,\beta\in\mathbb{C}$.
	
	This shows that the subspace $\langle \overline{s}+\mathfrak{m}^2\rangle$ is an eigenspace for every transformation $\rho(\xi)\in\mathfrak{gl}(V)$, where  $\xi\in\operatorname{Der}(A)$ and  $P_A(\overline{s}+\mathfrak{m}^2) = 0$. So, $A$ is not of full null-index. But we can notice that in this case $\operatorname{Der}(A)$ is solvable, because in the basis $\{\overline{t}, \overline{s}, \overline{t}^2, \overline{s}^2\}$ of $\mathfrak{m}$ all the operators $\xi\in\mathfrak{g}$ have lower triangular matrices.
\end{example}

Here is a key lemma on Gorenstein algebras of fulll null-index.
 
 \begin{lemma}\label{lem1}
 	Let $A$ be a Gorenstein algebra of full null-index and $\mathfrak{k}$ be a Lie subalgebra of   $\mathfrak{g} = \operatorname{Der}(A)$ that acts trivially on $\operatorname{Soc}(A)$.Then $\mathfrak{k}$ is solvable.
 \end{lemma}
 \begin{proof}
 	Let $r$ be the index of nilpotency of $\mathfrak{m}$. Since $\mathfrak{m}^r\neq 0$ and $\dim\operatorname{Soc}(A) = 1$ we have $\mathfrak{m}^r = \operatorname{Soc}(A)$.
 	
 Let $\rho_k$ be the restriction of the  representation $\rho$ to $\mathfrak{k}$.  Put $\mathfrak{D} = \operatorname{ker}(\rho_k)$, so that
$$\mathfrak{D} = \{\eta\in\mathfrak{k}\colon \eta(m)\in\mathfrak{m}^2,\forall m\in\mathfrak{m}\} .$$
Obviously, $\mathfrak{D}$ is an ideal of $\mathfrak{k}$. Fore  $i\in\mathbb{Z}_{\geqslant 0}$ let us introduce $$\mathfrak{D}^{(i)}=[\mathfrak{D}^{(i-1)}, \mathfrak{D}^{(i-1)}],$$  where $\mathfrak{D}^{(0)}=\mathfrak{D}$. Then every transformation $\eta\in\rho_k(\mathfrak{D}^{(i)})$ take $\mathfrak{m}$ to  $\mathfrak{m}^{i+2}$. So, since $\mathfrak{m}^{r+1} = 0$, $\mathfrak{D}$ is solvable.

Now it suffices to prove that $\rho_k(\mathfrak{k})$ is solvable. In fact we will prove that  $\rho_k(\mathfrak{k})$ is nilpotent.

Given a derivation $\xi\in\mathfrak{k}$ for $m\in\mathfrak{m}$ and $v = \pi(m)$ we have
$$0=P_A(v)\xi(s)=\xi(P_A(v)\cdot s) = \xi(m^r) = rm^{r-1}\xi(m) .$$
So, 
\begin{equation}\label{zero}
	m^{r-1}\xi(m) = 0,\quad \forall m\in\mathfrak{m}, \xi\in \mathfrak{k}.
\end{equation}
 Fix a scalar $\lambda\in\mathbb{C}$. Then from  identity (\ref{zero}) we can derive that
\begin{multline}\label{eq1}
	P_A(v+\lambda\rho_r(\xi)v)\cdot s = \\=(m+\lambda\xi(m))^r = m^r+\lambda^2\cdot \dfrac{r(r-1)}{2}\cdot m^{r-2}(\xi(m))^2 +... +\lambda^r(\xi(m))^r.
\end{multline}
On the other hand we have
\begin{equation}\label{eq2}
	P_A(v+\lambda\rho_r(\xi)v)\cdot s  = P_A(v)\cdot s + \lambda\operatorname{d}P_A({v})(\rho_k(\xi)v) \cdot s + \lambda^2F(v, \rho_{k}(\xi))\cdot s,
\end{equation}
where $F(v, \rho_{k}(\xi))$  is the linear combination of higher differentials of  the polynomial $P_A$. If we  compare two equalities (\ref{eq1}) and (\ref{eq2}) we will get that 
\begin{equation}\label{eq3}
	\operatorname{d}P_A({v})(\rho_k(\xi)v) =0,\quad \forall \xi\in\mathfrak{k}, v\in V.
\end{equation}

Consider an eigenvector $v$ of the linear operator $\rho_k(\xi)$, where $\xi\in\mathfrak{k}$, with eigenvalue $\lambda$.  Using equality (\ref{eq3}) we have
\begin{equation}\label{eq4}
	0=\operatorname{d}P_A({v})(\rho_k(\xi)v)  = \lambda\operatorname{d}P_A({v})(v)  = \lambda rP_A(v).
\end{equation}
The last equality in identity (\ref{eq4}) holds thanks to Euler's identity for homogeneous functions.    If we choose $v$ such that $P_A(v)\neq 0$ (here we use the fact that $\mathfrak{g}$ is algebra of full null-index), then we will get from (\ref{eq4}) that $\lambda = 0$. Consequently, the linear transformation $\rho_k(\xi) $ is nilpotent. 

Consider the operator $\zeta = \operatorname{ad}(\rho_k(\xi))\in\mathfrak{gl}(\mathfrak{g})$,  where $$\operatorname{ad}\colon\mathfrak{k}\rightarrow\mathfrak{gl}(\mathfrak{g})$$ is the adjoint representation of $\mathfrak{k}$.  Then $\zeta$ is nilpotent too.  It means that the Lie algebra $\mathfrak{k}$ is nilpotent, which follows from Engel's criterion.
\end{proof}
\section{Solvability of derivation algebra over Gorenstein algebra of full null-index}\label{sec2}
Now let $A$ be an arbitrary Gorenstein algebra again  with  the maximal ideal $\mathfrak{m}$,  such that $\mathfrak{m}^{r+1} = 0, \mathfrak{m}^r\neq 0$. Put again $\mathfrak{g} = \operatorname{Der}(A)$. Let $s$ be the generator of $\operatorname{Soc}(A)$, i.e $\operatorname{Soc}(A) = \langle s\rangle$.

\begin{lemma}\label{lem2}
	If  $\mathfrak{g}$ is not solvable then for an arbitrary $\xi$ from the irrelevant subspace $\mathfrak{g}_{*}$  the equality
	$\xi(s) = 0$ holds.
\end{lemma}
\begin{proof}
	Firstly, there are three derivations $e, f, h\in\mathfrak{g}$, that the triple $T=(e, f, h)$ is $\mathfrak{sl}_2$-triple, because $\mathfrak{g}$ is unsolvable. Notice that every element of $\mathfrak{sl}_2$-triple $T$ preserves the socle  $\operatorname{Soc}(A)$. So, $\operatorname{Soc}(A)$ is one-dimensional Lie $\mathfrak{a}$-module, where $\mathfrak{a} = \langle e, f, h\rangle \subset \mathfrak{g}$. Since $\mathfrak{a}\simeq \mathfrak{sl}_2$ the action of $\mathfrak{a}$ on $\operatorname{Soc}(A)$ is trivial. 
	
	Consider the grading (\ref{grad}) of $\mathfrak{g}$. If we take a derivation $\xi\in\mathfrak{g}^{k}$ then we will have
	\begin{equation}\label{key0}
		h(\xi(s)) = [h, \xi](s) + \xi(h(s)) = k\xi(s);
	\end{equation}  
	\begin{equation}\label{conj}
		e(\xi(s)) = [e, \xi](s) + \xi(h(s)) = 0.
	\end{equation}  
	Further if we take $\xi\in\mathfrak{g}^{k-1}$, then $[e, \xi] = 0$ and, consequently, $e(\xi(s)) = 0$. So, in the general case if we assume that derivation $e$ acts trivially on the $\xi(s)$, where  $\xi\in\mathfrak{g}^{i}$ or $\xi\in \mathfrak{g}^{i+1}$, then 
	the computation (\ref{conj}) will show us that  for an arbitrary derivation $\xi\in \mathfrak{g}^{i-1}$  derivation $e$ acts trivially on $\xi(s)$ too, because $[e, \xi]\in\mathfrak{g}^{i+1}$. Consequently,
	\begin{equation}\label{key1}
		 e(\xi(s)) = 0, \quad\forall\xi\in \mathfrak{g}^{l}, -k\leqslant l\leqslant k.
	\end{equation}
	 Analogically, if the consider the action of element $f$ on the $\xi(s)$, then considering a derivation $\xi $ step by step from the case $\xi\in\mathfrak{g}^{-k}$  to the case $\xi\in\mathfrak{g}^k$, we will get that
	 \begin{equation}\label{key2}
	 	  f(\xi(s)) = 0, \quad\forall\xi\in \mathfrak{g}^{l}, -k\leqslant l\leqslant k.
 \end{equation}
	Three indentities (\ref{key0}),  (\ref{key1}) and (\ref{key2}) give us the fact that for an arbitrary $\xi\in\mathfrak{g}_{*}$ the subspace $U = \langle \xi(s)\rangle \subset A$ is one-dimensional Lie $\mathfrak{a}$-module, so as it will be higher $\mathfrak{a}$ must act trivially on $U$. But for $\xi\in \mathfrak{g}^l$, where $-k\leqslant l\leqslant k, l\neq 0$, we have
		\begin{equation}\label{key3}
		h(\xi(s)) = [h, \xi](s) + \xi(h(s)) = l\xi(s).
	\end{equation}  
	So, since $l\neq 0$ it must be $\xi(s) = 0$. The lemma is proved. 
\end{proof}
Now it is time to formulate and prove the main result of the paper.
\begin{theorem}\label{main}
	Let $A$ be a Gorenstein algebra of full null-index. Then the Lie algebra $\operatorname{Der}(A)$ is solvable.
\end{theorem}
\begin{proof}
	Assume that $\mathfrak{g}=\operatorname{Der}(A)$ is not solvable. Then it follows from Lemma \ref{lem2}, that irrelevant subspace $\mathfrak{g}_{*}$ acts trivially on the socle $\operatorname{Soc}(A) = \langle s\rangle$. 
	
	Consider two different derivations $\xi\in \mathfrak{g}^l, \eta\in \mathfrak{g}^{-l}$, where $l\neq 0$. It is clear that $[\xi, \eta]\in\mathfrak{g}^0$. Then $\xi(s) = \eta(s) = 0$ and $[\eta, \xi](s) = 0$. Let $\mathfrak{g}^0_{*}$ denote the Lie subalgebra of  $\mathfrak{g}$, which is generated by all the commutators $[\xi, \eta]$, where $\xi\in \mathfrak{g}^l, \eta\in \mathfrak{g}^{-l}$, $l\neq 0$. Consider the subspace $\mathfrak{g}^{*} = \mathfrak{g}_{*}\oplus\mathfrak{g}^0_{*}$ in $\mathfrak{g}$. It is not difficult to see that $\mathfrak{g}^{*}$ is  a Lie subalgebra in $\mathfrak{g}$, that acts trivially on the socle $\operatorname{Soc}(A)$. Consequently, from Lemma \ref{lem1} it follows that $\mathfrak{g}^{*}$ is solvable. But the Lie algebra $\mathfrak{a} = \langle e, f, h\rangle\simeq\mathfrak{sl}_2$ lies in $\mathfrak{g}^{*}$. So we get a contradiction, which means that $\mathfrak{g}$ is solvable.
	\end{proof}
	Theorem \ref{main} shows that the Lie algebra of derivations of a Gorenstein algebra of full null-index is solvable, indicating a restricted structure. This leads to a natural question: is the Lie algebra of derivations $\operatorname{Der}(A)$ solvable for every Gorenstein algebra $A$? Example \ref{ex2} demonstrates that some Gorenstein algebras not of full null-index do indeed have solvable derivation algebras. However, in general, the derivation algebra of an arbitrary Gorenstein algebra need not be solvable. The following example illustrates this.
	\begin{example}
	Put $A = \mathbb{C}[x_1, x_2, x_3]/I$, where $I = (x_1x_2, x_1x_3, x_2x_3, x_1^2 -  x_2^2, x_1^2 - x_3^2)$. Then
	$$\mathfrak{m} = \langle \overline{x}_1, \overline{x}_2, \overline{x}_3, \overline{x}_1^2\rangle,\quad\mathfrak{m}^2=\operatorname{Soc}(A) = \langle \overline{x}_1^2\rangle, \quad V\backsimeq\langle \overline{x}_1, \overline{x}_2, \overline{x}_3\rangle,\quad r=2.$$
	The matrix $X$ of an arbitrary derivation $\xi\in \operatorname{Der}(A) = \mathfrak{g}$ in the basis $\{\overline{x}_1,\overline{x}_2,\overline{x}_3, \overline{x}_1^2\}$ has
	$$X = \begin{pmatrix}
		\alpha&\alpha_2^1&\alpha_3^1&0\\
		-\alpha_2^1&\alpha&\alpha^2_3&0\\
		-\alpha_3^1&-\alpha_3^2&\alpha&0\\
		\beta_1&\beta_2&\beta_3&2\alpha
	\end{pmatrix}$$
	for some $\alpha, \alpha_i^j,\beta_i\in\mathbb{C}$. One can easily see that there is a subalgebra $$\mathfrak{s} := \{X\in\mathfrak{g}: \alpha=\beta_1=\beta_2=\beta_3 = 0\}$$in $\mathfrak{g}$, which is isomorphic to $\mathfrak{so}_3$. So, here $\operatorname{Der}(A)$ is not solvable.
\end{example}

The author hopes that Theorem \ref{main} can eventually be proved for a wider subclass of Gorenstein algebras. This work constitutes a first step toward that goal.

\end{document}